\documentclass[11pt,reqno]{amsart}
\usepackage[T1]{fontenc}
\usepackage[utf8]{inputenc}
\usepackage{lmodern}
\usepackage{amsmath,amssymb,graphicx,microtype}
\usepackage[scale=0.7]{geometry}
\usepackage[colorlinks=true,allcolors=blue]{hyperref}

\newtheorem{theorem}{Theorem}[section]
\newtheorem{lemma}[theorem]{Lemma}
\newtheorem{proposition}[theorem]{Proposition}

\theoremstyle{definition}
\newtheorem{definition}[theorem]{Definition}

\newcommand{\R}{\mathbb R}
\newcommand{\E}{\mathcal E}
\newcommand{\B}{\mathcal B}
\DeclareMathOperator{\conv}{conv}
\DeclareMathOperator{\diam}{diam}
\DeclareMathOperator{\intt}{int}
\newcommand{\norm}[1]{\lVert#1\rVert}

\title{A new family of constant width bodies obtained from Reuleaux polyhedra}

\author[M. G. Mercado-Flores]{Marcela G. Mercado-Flores}
\address[M. G. Mercado-Flores]{Centro de Ciencias Matemáticas, UNAM Campus Morelia, Morelia, Mexico}
\email{mmercado@matmor.unam.mx}

\author[E. Roldán-Pensado]{Edgardo Roldán-Pensado}
\address[E. Roldán-Pensado]{Centro de Ciencias Matemáticas, UNAM Campus Morelia, Morelia, Mexico}
\email{eroldan@matmor.unam.mx}


\date{}
\begin{document}

\begin{abstract}
We give a construction of a body of constant width from a unit Reuleaux polyhedron that preserves all its symmetries. As opposed to the Meissner surgery, our construction modifies both edges of each dual pair and requires no choice between them. For the Reuleaux tetrahedron it recovers a previously known body of constant width with six elliptic edges and tetrahedral symmetry.
\end{abstract}

\maketitle

\section{Introduction}

A convex body has constant width if the distance between any two parallel supporting planes is the same in all directions. In the plane, Reuleaux polygons are the classic examples of such bodies, but their three-dimensional analogues, Reuleaux polyhedra, do not have constant width. In the case of the Reuleaux tetrahedron, Meissner \cite{Mei18} showed that one can round one edge from each dual pair to obtain a body of constant width. This method was later extended to general Reuleaux polyhedra by Montejano and Roldán-Pensado \cite{MR17}, but the choice of which edge to round in each pair is not canonical and may destroy some of the symmetries of the original polyhedron.
For the Reuleaux tetrahedron, an alternative modification using focal conics leads to the so-called peabodies of constant width, introduced by Arellano, Montejano, and Oliveros \cite{AMO23} (see also \cite{AMO24}). However, this construction only works in the tetrahedral setting. For a comprehensive introduction to convex bodies, see \cite{MMO19}.

In this paper, we present a different modification that applies to all Reuleaux polyhedra while preserving all of their symmetries.
Given a Reuleaux polyhedron of radius one, we replace both circular edges of each dual pair by elliptic arcs and intersect all unit balls centered at these arcs.
The resulting body retains the original vertices, contains the elliptic arcs in its boundary, and preserves the symmetry group of the underlying Reuleaux polyhedron.

Throughout this paper, a \emph{Reuleaux polyhedron} is a three-dimensional standard ball polyhedron whose vertices are precisely the centers of its defining balls. We normalize the radii to one. Standardness means that two distinct faces intersect in at most one edge or one vertex. We use the vertex--face duality induced by the centers. In particular, the edges come in dual pairs. For further details on ball polyhedra and duality, we refer to \cite{MR17} (see also \cite{KMP10,MPRR20}).

Let $R$ be such a polyhedron with vertex set $V$. For each dual pair of edges $e,e^*$, denote their endpoints by $A_-,A_+$ and $B_-,B_+$, respectively. Then the four distances of the form $\norm{A_\varepsilon-B_\delta}$ are equal to one. In Section~\ref{sec:arcs}, we associate elliptic arcs $\gamma_e,\gamma_{e^*}$ to each dual pair.
Set
\[
 \Gamma(R)=\bigcup_e\gamma_e,
 \quad
 \E(R)=\bigcap_{p\in\Gamma(R)}B(p,1),
\]
where the union is taken over all circular edges of $R$. Figure~\ref{fig:bodies} shows an example of $\Gamma(R)$ and $\E(R)$ for a specific $R$.

\begin{theorem}\label{thm:main}
Let $R$ be a Reuleaux polyhedron. Then $\E(R)$ has the following properties:
\begin{itemize}
    \item $\E(R)$ has constant width one
    \item $V\subset\Gamma(R)\subset\partial \E(R)$ and $\E(R)\subset R$.
    \item Every point of $\partial \E(R)\setminus\Gamma(R)$, as well as the midpoints of the arcs $\gamma_e$, are smooth.
    \item Every diameter of $\E(R)$ has an endpoint in $\Gamma(R)$.
    \item Every isometry preserving $R$ also preserves $\E(R)$.
\end{itemize}
\end{theorem}

\begin{figure}
    \centering
    \includegraphics[scale=.8]{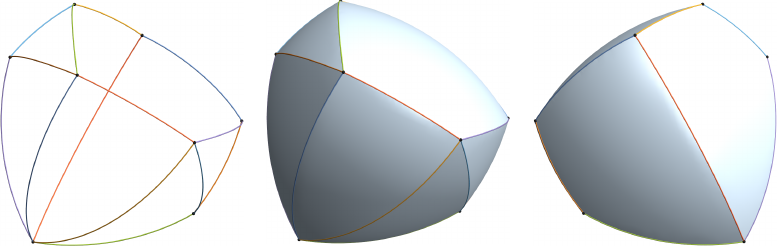}
    \caption{A constant-width body derived from a 7-vertex Reuleaux polyhedron with elliptical opposite edges. First its ellpitical edges, alongside the solid body and a 180-degree rotated view. }
    \label{fig:bodies}
\end{figure}

The tetrahedral case of this construction was studied previously through different viewpoints. Arman, Bondarenko, Prymak, and Radchenko constructed a tetrahedrally symmetric body of constant width as a projection of a four-dimensional body \cite{ABPR25}. Later, Mercado-Flores, Raggi, and Roldán-Pensado obtained the same body from an alternative four-dimensional construction and described it as an intersection of balls centered on six elliptic arcs \cite{MRR25}.  Our construction includes this case and extends it to Reuleaux polyhedra.

The remainder of this paper is organized as follows. In Section~\ref{sec:arcs}, we construct the elliptic arcs and compare them with the original circular edges and the Meissner modification. In Section~\ref{sec:skeleton}, we prove that the elliptic skeleton has diameter one and establish the properties needed for the proof of Theorem~\ref{thm:main}. In Section~\ref{sec:smoothness}, we use these properties to show that every boundary point outside the elliptic skeleton has a unique supporting plane. Finally, in Section~\ref{sec:width}, we deduce constant width from this smoothness property and complete the proof of Theorem~\ref{thm:main}. We also describe the nonsmooth points and the surface patches. The final section includes some final remarks in which we discuss the tetrahedral case, and explain how to combine elliptic and Meissner modifications while preserving the symmetries of the original polyhedron.

\section{A pair of elliptic arcs}\label{sec:arcs}

To modify a Reuleaux polyhedron $R$, we replace each dual pair of circular edges with a pair of elliptic arcs.

Fix four vertices $A_-,A_+,B_-,B_+$ forming a pair of dual edges in $R$. This configuration satisfies
\[
\norm{A_\varepsilon-B_\delta}=1\quad\text{for }\varepsilon,\delta\in\{-,+\},
\quad \norm{A_+-A_-}=a,\quad\norm{B_+-B_-}=b,
\]
where $0<a,b\leq 1$. These conditions place the line $A_-A_+$ in the perpendicular bisector plane of $B_-B_+$, and vice versa. Consequently, the directions of $A_-A_+$ and $B_-B_+$ are orthogonal, and their midpoints lie on their common perpendicular segment.

In our setting, consider a plane $H$ parallel to both cross-segments $A_+B_+$ and $A_-B_-$. The common perpendicular to these two segments intersects $H$ orthogonally at a point $O'$. When projecting the four vertices orthogonally onto $H$, their projections form an isosceles trapezoid, and both $A_+B_+$ and $A_-B_-$ maintain their unit lengths. To ensure a constant distance between corresponding points along the new edge curves, we seek spatial curves whose orthogonal projections onto $H$ are concentric circular arcs centered at $O'$ (see Figure~\ref{fig:proj}).

\begin{figure}
    \centering
    \includegraphics{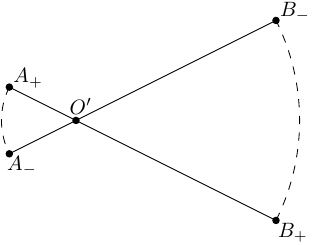}
    \caption{Orthogonal projection onto the plane $H$. The vertices $A_-,A_+,B_-,B_+$ are projected into a isosceles trapezoid. The projections of the desired elliptic arcs should form concentric circular arcs centered at $O'$.}s
    \label{fig:proj}
\end{figure}

To simplify future computations, we define the parameters
\begin{equation}\label{eq:parameters}
c=\frac{a}{\sqrt{a^2+b^2}},\quad
d=\frac{b}{\sqrt{a^2+b^2}},\quad
t_0=\arcsin\frac{\sqrt{a^2+b^2}}{2},\quad h=\cos t_0.
\end{equation}
Note that $c^2+d^2=1$ and $0<t_0\leq\pi/4$. The distance between the midpoints of $A_-A_+$ and $B_-B_+$ is $\sqrt{1-(a^2+b^2)/4}=h$. After applying a suitable isometry, we may align the midpoint of their common perpendicular segment with the origin, and place the four points at
\begin{equation}\label{eq:vertices}
 A_\pm=(\pm a/2,0,hc^2),\quad
 B_\pm=(0,\pm b/2,-hd^2).
\end{equation}

With this positioning, the $z$-coordinates ensure that the common perpendiculars of the lines $A_+B_+$ and $A_-B_-$ pass through the origin in directions $(d,c,0)$ and $(d,-c,0)$. This allows the curves to project to circular arcs under orthogonal projection onto $H$.

\begin{definition}\label{def:arcs}
The elliptic arcs associated with the dual edge pair $A=A_-A_+$ and $B=B_-B_+$ are parameterized by
\begin{equation}\label{eq:arcs}
\gamma_A(t)=(c\sin t,0,c^2\cos t),\quad
\gamma_B(t)=(0,d\sin t,-d^2\cos t),\quad
\text{for }-t_0\leq t\leq t_0.
\end{equation}
Slightly abusing notation, we also denote their images by $\gamma_A$ and $\gamma_B$ (see Figure~\ref{fig:cilindros}).
\end{definition}

\begin{figure}
    \centering
    \includegraphics{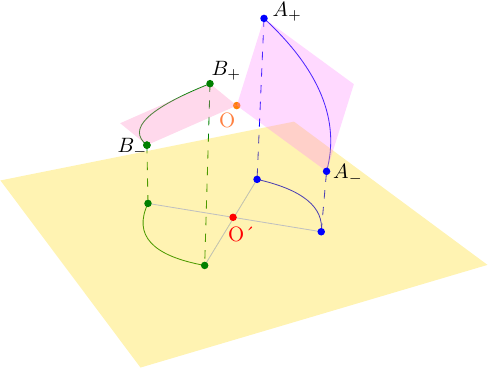}
    \caption{The two elliptical arcs and their projection to $H$.}
    \label{fig:cilindros}
\end{figure}
\begin{lemma}\label{lem:distance}

Evaluating at $t=\pm t_0$ shows that the endpoints of $\gamma_A$ and $\gamma_B$ are the vertices in \eqref{eq:vertices}. Reversing either pair of endpoints merely reverses its parametrization. Interchanging the two pairs interchanges the two arcs. Thus, the curves $\gamma_A$ and $\gamma_B$ are uniquely determined by the dual edge pair independently of labeling choices.

This also has a geometric interpretation using revolving cylinders. The squared distance from a point $(x,y,z)$ to the common perpendicular direction $(d,c,0)$ is
\begin{align}
    \norm{(x,y,z)}^2-((x,y,z)\cdot(d,c,0))^2&=x^2+y^2+z^2-d^2x^2-c^2y^2-2cdxy\nonumber\\
    &=c^2x^2+d^2y^2-2cdxy+z^2\nonumber\\
    &=(cx-dy)^2+z^2.\label{eq:distcilinder}
\end{align}
The cylinders of radii $c^2$ and $d^2$ around this line intersect the coordinate planes $y=0$ and $x=0$ in the ellipses containing the arcs $\gamma_A$ and $\gamma_B$. Using the second common perpendicular direction $(d,-c,0)$ yields the same pair of ellipses.

For $s,t\in[-t_0,t_0]$,
\begin{equation}\label{eq:distance}
\norm{\gamma_A(t)-\gamma_B(s)}^2=1-c^2d^2(\cos t-\cos s)^2.
\end{equation}
Moreover, $\diam(\gamma_A)=a$ and $\diam(\gamma_B)=b$. In particular, $\diam(\gamma_A\cup\gamma_B)=1$, and equality $\norm{\gamma_A(t)-\gamma_B(s)}=1$ holds if and only if $s=t$ or $s=-t$.
\end{lemma}
\begin{proof}
Expanding the squared distance gives
\begin{align*}
    \norm{\gamma_A(t)-\gamma_B(s)}^2 &=\norm{(c\sin t,0,c^2\cos t)-(0,d\sin s,-d^2\cos s)}^2\\
    &=c^2\sin^2 t+d^2\sin^2 s+(c^2\cos t+d^2\cos s)^2\\
    &=1-c^2\cos^2 t-d^2\cos^2 s+(c^2\cos t+d^2\cos s)^2\\
    &=1-c^2(1-c^2)\cos^2 t-d^2(1-d^2)\cos^2 s+ 2c^2d^2\cos t\cos s\\
    &=1-c^2d^2(\cos t-\cos s)^2.
\end{align*}
which gives \eqref{eq:distance}.
Since cosine is strictly decreasing on $[0,t_0]$, $\cos t=\cos s$ is equivalent to $|t|=|s|$. This shows that $\norm{\gamma_A(t)-\gamma_B(s)}=1$ if and only if $s=\pm t$.

For two points on $\gamma_A$, we use the identities
\begin{align*}
 \sin s-\sin t&=2\sin\left(\frac{s-t}{2}\right)\cos\left(\frac{s+t}{2}\right)\\
 \cos s-\cos t&=2\sin\left(\frac{s+t}{2}\right)\sin\left(\frac{s-t}{2}\right).
\end{align*}
Then, since $\sin$ is strictly increasing on $[0,t_0]$ and $c^2<1$,
\begin{align*}
    \norm{\gamma_A(s)-\gamma_A(t)}^2&=\norm{(c\sin s,0,c^2\cos s)-(c\sin t,0,c^2\cos t)}^2\\
    &=c^2(\sin s-\sin t)^2+c^4(\cos s-\cos t)^2\\
    &=4c^2\sin^2\left(\frac{s-t}{2}\right)\left(\cos^2\left(\frac{s+t}{2}\right)+c^2\sin^2\left(\frac{s+t}{2}\right)\right)\\
    &\leq4c^2\sin^2t_0=a^2.
\end{align*}
The endpoints attain this bound, proving that $\diam(\gamma_A)=a$. The proof of $\diam(\gamma_B)=b$ is analogous. We conclude that the diameter of $\gamma_A\cup\gamma_B$ is indeed $1$.
\end{proof}

In the Reuleaux polyhedron, the edge connecting $A_-$ and $A_+$ is a circular arc. When performing a Meissner surgery to obtain a body of constant width, one either leaves this edge unmodified or replaces it with a surface of revolution. We show that our elliptic arc $\gamma_A$ lies strictly between the original Reuleaux edge and the middle curve of the modified edge.

Let $\alpha_A$ be the shorter circular arc from $A_-$ to $A_+$ on the intersection of the spheres $\partial B(B_-,1)\cap\partial B(B_+,1)$, located on the side opposite from the segment $B_-B_+$. This is precisely the circular edge of the Reuleaux polyhedron. We define $\alpha_B$ analogously.

The Meissner modification replaces the edge with a surface of revolution generated by a unit circle. Its intersection with the the plane $y=0$ gives an arc $\mu_A$ of radius $1$. Another way of defining the Meissner modification is by intersecting all unit balls contained in $\alpha_B$.

Because the three curves share the endpoints $A_\pm = (\pm a/2, 0, hc^2)$ and are symmetric with respect to the $z$-axis, we can analyze their relative positions by comparing their heights in the $xz$-plane over the interval $x \in [-a/2, a/2]$.

\begin{description}
    \item[The Reuleaux arc] Since $\alpha_A$ lies on the intersection of the unit spheres centered at $B_-$ and $B_+$, its center is $(0,0,-hd^2)$ and its radius is $\sqrt{1-b^2/4}$. Its height function is
    \begin{equation*}
        Z_R(x) = -hd^2+\sqrt{1-\frac{b^2}{4}-x^2}.
    \end{equation*}
    \item[The elliptic arc] The equation for $\gamma_A$ is immediatley obtained from \eqref{eq:distcilinder}. Setting $y=0$ gives a height function
    \begin{equation*}
        Z_E(x) = c\sqrt{c^2-x^2}.
    \end{equation*}
    \item[The Meissner lower bound] The center of $\mu_A$ is the midpoint of the arc $\alpha_B$, which has coordinates $(0,0,hc^2-\sqrt{1-a^2/4})$. Therefore, its height function is
    \begin{equation*}
        Z_M(x) = hc^2-\sqrt{1-\frac{a^2}{4}}+\sqrt{1-x^2}.
    \end{equation*}
\end{description}
These three functions represent the heights of $\alpha_A$, $\gamma_A$, and the central section $\mu_A$ of the surface obtained by rounding this edge in the Meissner surgery (see Figure~\ref{fig:heights}).

\begin{figure}
    \centering
    \includegraphics[scale=0.58]{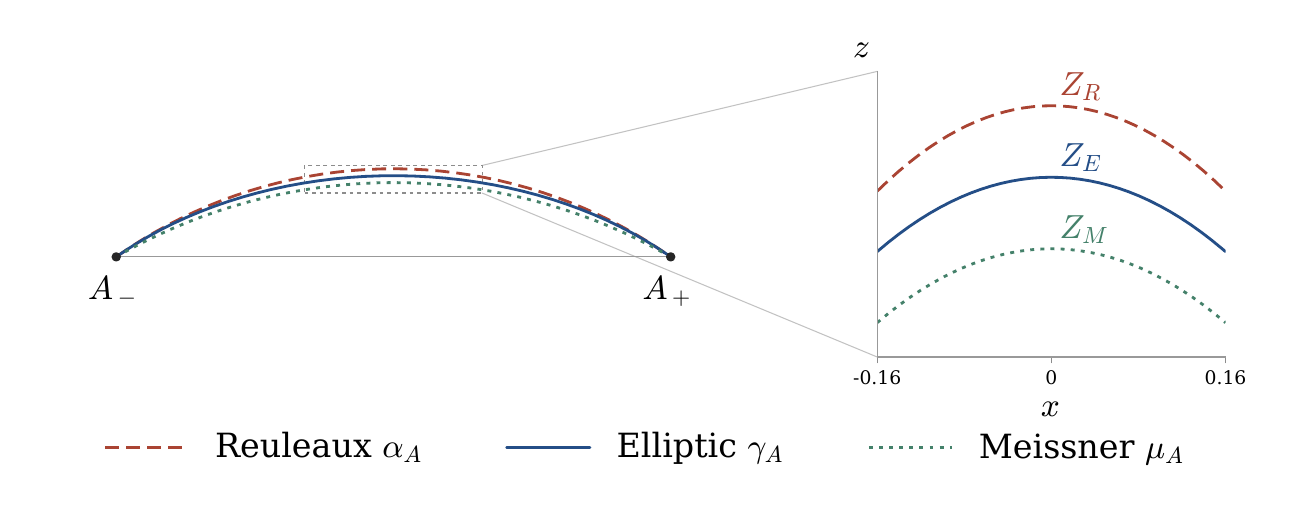}
    \caption{Relative position of the elliptical edge compared to the circular edge of the Reuleaux polyhedron and the modified edge of the Meissner body.}
    \label{fig:heights}
\end{figure}

\begin{lemma}\label{lem:comparison}
If $|x|<a/2$ then
\[
 Z_M(x)<Z_E(x)<Z_R(x).
\]
These three functions agree at the endpoints $x=\pm a/2$.
In other words, every interior point of the elliptic arc $\gamma_A$ is contained in the convex hull of $\alpha_A$ and is strictly outside $\mu_A$.
\end{lemma}

\begin{proof}
For the upper bound, Lemma~\ref{lem:distance} implies that any interior point $\gamma_A(t)$ is at a distance strictly less than one from $B_+$ and $B_-$. When restricted to the plane $y=0$ and above the line $z=hc^2$, this geometric condition forces $Z_E(x)<Z_R(x)$.

For the lower bound, we compute the derivative of their difference
\[
 (Z_E-Z_M)'(x) = x\left(\frac{1}{\sqrt{1-x^2}}-\frac{c}{\sqrt{c^2-x^2}}\right).
\]
Since both functions are even and coincide at $a/2$, it is enough to show that this is negative for $0<x<a/2$. To prove this it is enough to show that
\[
\frac{1}{1-x^2}<\frac{c^2}{c^2-x^2}.
\]
This is equivalent to $c^2-x^2 < c^2-c^2x^2$, which simplifies to $0 < x^2(1-c^2)$. Since $c<1$, this last inequality is true for $x>0$. So the function $Z_E-Z_M$ must be strictly positive on the interval $(-a/2,a/2)$.
\end{proof}

\section{The elliptic skeleton}\label{sec:skeleton}

In this section, we study the structure of the elliptic skeleton, defined as the union
\[
 \Gamma(R)=\bigcup_e\gamma_e
\]
over all circular edges $e$ of the Reuleaux polyhedron $R$. We prove that it is strictly contained within $R$, and analyze how each elliptic arc interacts locally with the regions modified during classic Meissner surgery.

For a compact set $S\subset\R^3$, write
\[
 \B(S)=\bigcap_{q\in S}B(q,1).
\]
Notice that $\B(S)=\B(\conv S)$, by convexity of the norm.

For an edge $e$ of $R$ with dual edge $e^*$, consider the set $D_e$ of points in $R$ that are removed when applying the Meissner modification to $e$. Formally, we define
\[
 D_e=R\setminus\B(\alpha_{e^*}).
\]
This set contains the relative interior of the circular edge $\alpha_e$ and has $\mu_e$ on its boundary.

We state some well-known properties of $D_e$ that are important in the study of Meissner polyhedra. These results can be found in \cite[Section 4]{MR17}, \cite[Lemmas 4.2, 4.4, and Proposition 4.5]{Hyn23}, and \cite[Section 2]{Hyn26}.

\begin{proposition}\label{prop:surgery}
For dual edges $e=A_-A_+$ and $e^*=B_-B_+$, the set $D_e$ is contained in the open dihedral region beyond the planes $A_-A_+B_-$ and $A_-A_+B_+$, on the sides opposite the other $B$ vertex. Its intersection with $\partial R$ meets only the two faces centered at $B_-$ and $B_+$. If $e$ and $f$ belong to different dual pairs of edges, then $D_e\cap D_f=\varnothing$.
\end{proposition}

For clarity, in the coordinates \eqref{eq:vertices}, the dihedral region containing $D_e$ is expressed as
\begin{equation}\label{eq:dihedral}
 z>hc^2+\frac{2h}{b}|y|.
\end{equation}
For $D_{e^*}$, the corresponding condition is $z<-hd^2-(2h/a)|x|$. This implies that $D_e\cap D_{e^*}=\varnothing$ as well, so the sets $D_e$ are pairwise disjoint.

In the following lemma, we place our elliptic skeleton strictly inside the Reuleaux polyhedron $R$ and confirm that its total diameter equals one.

\begin{lemma}\label{lem:skeleton}
The elliptic skeleton satisfies
\[
 \diam\Gamma(R)=1,\quad \Gamma(R)\subset R,
\]
and every interior point of $\gamma_e$ belongs to $D_e\cap\intt R$.
\end{lemma}
\begin{proof}
Lemma~\ref{lem:comparison} immediately implies that $\gamma_e\subset D_e\cap\intt R$, so $\Gamma(R)\subset R$.

To compute the diameter, consider points $p\in\gamma_e$ and $q\in\gamma_f$. If $f=e$ or $f=e^*$, Lemma~\ref{lem:distance} bounds their distance by $1$. Otherwise, choose a Meissner body in which neither $\alpha_e$ nor $\alpha_f$ has been modified. Since this Meissner body contains $\gamma_e\cup\gamma_f$ and has diameter 1, we obtain $\norm{p-q}\leq 1$. The distances between the original vertices give the equality, establishing that $\diam\Gamma(R)=1$.
\end{proof}

Next, we examine how an elliptic arc modifies the region on the opposite side of the Reuleaux polyhedron. Specifically, we show that any point in $R$ at unit distance from an interior point of $\gamma_e$ must lie inside the opposite dual region $D_{e^*}$. In other words, a ball centered at at point in $\gamma_e$ can only cut off a part from $R$ contained in $D_{e^*}$.

\begin{lemma}\label{lem:diam-int}
Let $p$ be an interior point of $\gamma_e$, and let $q\in R$ be a point such that $\norm{p-q}=1$. Then $q\in D_{e^*}$.
\end{lemma}
\begin{proof}
By Lemma~\ref{lem:comparison}, $p$ is a non-trivial convex combination of distinct points in $\alpha_e$. If $q\in\B(\alpha_e)$, strict convexity of the unit ball would imply $\norm{p-q}<1$. Therefore, $q\notin\B(\alpha_e)$, which is equivalent to $q\in D_{e^*}$.
\end{proof}

\section{Smoothness and supporting planes}\label{sec:smoothness}

Given a point $p\in\partial\B(S)$, we call $p$ \emph{smooth} if $\B(S)$ has a unique supporting plane at $p$. We say that $q\in S$ is a \emph{supporting center} of $p$ if $\norm{p-q}=1$. Compactness of $S$ ensures that every boundary point has a supporting center.

We start by showing that supporting centers are enough to detect smoothness.

\begin{lemma}\label{lem:normal}
Suppose that $S$ is compact and $L=\B(S)$ has nonempty interior. A point $p\in\partial L$ is smooth if and only if it has a unique supporting center. If this center is $q$, the outward unit normal at $p$ is $p-q$.
\end{lemma}
\begin{proof}
Every supporting center $q$ gives a supporting plane with outward unit normal $p-q$, so distinct supporting centers give distinct supporting planes. Thus a smooth point has a unique supporting
center.

Conversely, suppose that $q$ is the unique supporting center of $p$. Let $v$ be a unit vector satisfying $(p-q)\cdot v<0$. We first show that $p+\varepsilon v$ lies in the interior of $L$ for all sufficiently small $\varepsilon>0$.

Choose a neighborhood $U$ of $q$ and $\delta_1>0$ such that $(p-s)\cdot v\le -\delta_1$ for every $s\in S\cap U$. For these centers,
\[
\norm{p+\varepsilon v-s}^2=\norm{p-s}^2+2\varepsilon(p-s)\cdot v+\varepsilon^2\le 1-2\varepsilon\delta_1+\varepsilon^2.
\]
On the other hand, $S\setminus U$ is compact and contains no supporting center of $p$. If it is nonempty then, for every $s\in S\setminus U$, there is a $\delta_2>0$ such that $\norm{p-s}\le 1-\delta_2$ and therefore
\[
\norm{p+\varepsilon v-s}\le 1-\delta_2+\varepsilon.
\]
For sufficiently small $\varepsilon>0$, these bounds are uniformly smaller than $1$. Thus $p+\varepsilon v$ lies in the interior of $L$.

It follows that any outward supporting normal $u$ at $p$ satisfies $u\cdot v\le 0$ whenever $(p-q)\cdot v<0$. This forces $u$ to be a positive multiple of $p-q$. Therefore $p$ is smooth, with outward unit normal $p-q$.
\end{proof}

We now describe the nonsmooth points of $\B(\gamma_B)$ in the region
$z>hc^2$, which contains the interior of $\gamma_A$ (see \eqref{eq:dihedral}). They are precisely the points of $\gamma_A$ other than its midpoint. Each of these points has two supporting centers on $\gamma_B$. At the midpoint, the two centers coincide.

\begin{lemma}\label{lem:unique}
Using the notation in \eqref{eq:parameters}--\eqref{eq:arcs}, let $p=(x,y,z)\in\partial\B(\gamma_B)$ satisfy $z>hc^2$. Then $p$ is nonsmooth if and only if
\[
p\in\gamma_A\setminus\{A_-,\gamma_A(0),A_+\}.
\]
More precisely, if $p=\gamma_A(t)$ with $0<|t|<t_0$, its supporting centers on $\gamma_B$ are exactly $\gamma_B(t)$ and $\gamma_B(-t)$. The only supporting center of $\gamma_A(0)$ is $\gamma_B(0)$.
\end{lemma}
\begin{proof}
Since $p\in\partial\B(\gamma_B)$, its supporting centers are precisely the points of $\gamma_B$ farthest from $p$, and their distance from $p$ is $1$. We therefore consider
\begin{align*}
\norm{p-\gamma_B(s)}^2
&=x^2+(y-d\sin s)^2+(z+d^2\cos s)^2\\
&=x^2+y^2+z^2+d^2-c^2d^2\cos^2s+2zd^2\cos s-2yd\sin s.
\end{align*}
For a fixed value $u=\cos s\in[h,1]$, the larger of the values corresponding to $s$ and $-s$ is
\[
F(u)=x^2+y^2+z^2+d^2-c^2d^2u^2+2zd^2u+2d|y|\sqrt{1-u^2}.
\]
This function is strictly concave, so it has a unique maximizer on $[h,1]$.

If $y\ne 0$, this maximizer determines a unique point of $\gamma_B$. Indeed, when $u<1$, the sign of $\sin s$ must be opposite to that of $y$, and when $u=1$, we have $s=0$. Thus $p$ has a unique supporting center and is smooth because of Lemma~\ref{lem:normal}.

Suppose now that $y=0$. The function $F$ is a quadratic whose maximum for arbitrary $u$ occurs at $u=z/c^2$. Since $z>hc^2$, there are two possibilities. If $z\ge c^2$, the unique maximum on $[h,1]$ occurs at $u=1$, so again the only supporting center is $\gamma_B(0)$ and $p$ is smooth. If
\[
hc^2<z<c^2,
\]
the maximum occurs at $u=z/c^2\in(h,1)$ and corresponds to two distinct points of $\gamma_B$. Substituting this value into $F(u)=1$ gives
\[
x^2+\frac{z^2}{c^2}=c^2.
\]
Together with $y=0$ and $hc^2<z<c^2$, equation \eqref{eq:arcs} then gives $p=\gamma_A(t)$ for some $0<|t|<t_0$. By Lemma~\ref{lem:distance}, its supporting centers are exactly $\gamma_B(t)$ and $\gamma_B(-t)$.

Conversely, Lemma~\ref{lem:distance} shows that every $\gamma_A(t)$ with $0<|t|<t_0$ has these two distinct supporting centers and is therefore nonsmooth. It also shows that the only supporting center of $\gamma_A(0)$ is $\gamma_B(0)$.
\end{proof}

We now apply this description to $K=\B(\Gamma(R))$ in order to show that every boundary point outside $\Gamma(R)$ is smooth. The results from the previous section allow us to reduce the analysis to a single pair of dual arcs

\begin{proposition}\label{prop:smooth}
Let $K=\B(\Gamma(R))$. Then $\Gamma(R)\subset\partial K$, $K\subset R$, and every point of $\partial K\setminus\Gamma(R)$ is smooth.
\end{proposition}

\begin{proof}
Lemma~\ref{lem:skeleton} implies $\Gamma(R)\subset K$. By \eqref{eq:distance}, every point of each elliptic arc has a point of its dual arc at distance $1$ and therefore belongs to $\partial K$. Since the vertices $V\subset\Gamma(R)$, we also have $K\subset\B(V)=R$. The convex hull $V$ is a nondegenerate polyhedron contained in $K$, so $K$ has nonempty interior.

Now fix $p\in\partial K\setminus\Gamma(R)$. Suppose that all its supporting centers are vertices. If this includes at least two distinct vertices, then $p$ would be on an edge or vertex of $R$. But the vertices are the only points on the edges of $R$ that are in $K$ which contradicts that $p\notin\Gamma(R)$. Lemma~\ref{lem:normal} therefore shows that $p$ is smooth.

Otherwise, a supporting center of $p$ lies in the interior of some elliptic arc $\gamma_f$. By Lemma~\ref{lem:diam-int}, $p\in D_{f^*}$. The disjointness of the sets $D_e$ shows that every supporting center that is not a vertex lies on $\gamma_f$. If a vertex is a supporting center, then $p\in\partial R$. By Proposition~\ref{prop:surgery}, that vertex must be an endpoint of $f$. Thus all supporting centers of $p$ lie on $\gamma_f$.

Using the coordinates we defined for $f^*=A_-A_+$ and $f=B_-B_+$, \eqref{eq:dihedral} implies $z>hc^2$. Moreover, $p\in\partial\B(\gamma_f)$, since $p\in K$ and has a supporting center on $\gamma_f$. Lemma~\ref{lem:unique} shows that two distinct supporting centers would imply $p\in\gamma_{f^*}$, which contradicts the definition $p$. Hence $p$ has a unique supporting center, and by Lemma~\ref{lem:normal}, it is smooth.
\end{proof}

\section{Constant width and further properties}\label{sec:width}

The smoothness from the previous section is what is needed to prove constant width. The following lemma gives the general
argument and also describes the diameters of the resulting body.

\begin{lemma}\label{lem:criterion}
Let $S\subset\R^n$ be compact set of diameter $1$ and suppose that
$L=\B(S)$ has nonempty interior. Assume that $S\subset\partial L$
and that every point of $\partial L\setminus S$ is smooth.
Then $L$ has constant width one, and every diameter of $L$ has
an endpoint in $S$.
\end{lemma}

\begin{proof}
Since $L$ is the intersection of unit balls, it is strictly convex. Let $p,r$ be the endpoints of a diameter of $L$. If both belong to $S$, then $\norm{p-r}\le 1$. Otherwise, we may assume that $p\notin S$, so $p$ is smooth. Since $L\subset B(r,\norm{p-r})$, the vector $p-r$ is an outward supporting normal at $p$. By Lemma~\ref{lem:normal}, the outward unit normal at $p$ is $p-q$, where $q\in S$ is its unique supporting center. Thus $q$ and $r$ lie on the same ray from $p$ into $L$. Since both belong to $\partial L$, strict convexity forces $q=r$, so $\norm{p-r}=1$. It follow that every diameter has an endpoint in $S$ and that $\diam L=1$.

Finally, $\diam L=1$ implies $L\subset\B(L)$, and $S\subset L$ implies
\[
\B(L)\subset\B(S)=L.
\]
Therefore $L=\B(L)$. By the characterization of constant-width bodies \cite[Theorem 2.11]{Hyn23}, $L$ has constant width one.
\end{proof}

\subsection{The nonsmooth boundary points}
Proposition~\ref{prop:smooth} shows that all nonsmooth boundary
points of $\E(R)$ lie on the elliptic skeleton. We now determine
which points of the skeleton are nonsmooth.

Note that the point $\gamma_A(0)$ is the midpoint of $\gamma_A$. Let $T$ be the set of midpoints of all the elliptical arcs.

\begin{proposition}\label{prop:singular}
The set of nonsmooth boundary points of $\E(R)$ is exactly
$\Gamma(R)\setminus T$.
In particular, the midpoint of each elliptic arc is smooth,
whereas every other interior point of that arc is nonsmooth.
\end{proposition}

\begin{proof}
Consider a dual pair $e=A_-A_+$ and $e^*=B_-B_+$. If $0<|t|<t_0$, Lemma~\ref{lem:distance} gives two distinct supporting centers of $\gamma_A(t)$, namely $\gamma_B(t)$ and $\gamma_B(-t)$. Thus $\gamma_A(t)$ is nonsmooth by Lemma~\ref{lem:normal}.

Now let $p=\gamma_A(0)$. By Lemma~\ref{lem:skeleton}, $p\in D_e\cap\intt R$, so no vertex of $R$ is a supporting center of $p$. If a supporting center lies in the interior of an arc $\gamma_f$, Lemma~\ref{lem:diam-int} gives $p\in D_{f^*}$. Since the regions $D_e$ are pairwise disjoint, we must have $f=e^*$. All supporting centers of $p$ therefore lie on $\gamma_B$. By Lemma~\ref{lem:unique}, the only such center is $\gamma_B(0)$, so $p$ is smooth.

At each vertex, the centers of the incident faces give distinct supporting centers, so the vertex is nonsmooth.

Finally, Proposition~\ref{prop:smooth} shows that every boundary point outside $\Gamma(R)$ is smooth.
\end{proof}

Although each arc midpoint is smooth, every neighborhood of it
contains nonsmooth points of the same arc.

\subsection{Proof Theorem~\ref{thm:main}}
The proof of our main theorem is now a direct consequence of the previous results.

Lemma~\ref{lem:skeleton} and Proposition~\ref{prop:smooth} verify the hypotheses of Lemma~\ref{lem:criterion} for $S=\Gamma(R)$ and $L=\E(R)$. Thus $\E(R)$ has constant width one, and every diameter has an endpoint in $\Gamma(R)$. Proposition~\ref{prop:smooth} also gives
\[
V\subset\Gamma(R)\subset\partial\E(R), \quad \E(R)\subset R,
\]
and smoothness outside $\Gamma(R)$. Smoothness at the arc midpoints follows from Proposition~\ref{prop:singular}.

For the symmetry assertion, an isometry preserving $R$ permutes its vertices and faces, and therefore its dual edge pairs. The construction in Definition~\ref{def:arcs} depends only on the geometry of each unordered dual pair, so the isometry preserves $\Gamma(R)$ and its ball intersection $\E(R)$. \hfill $\qed$

\subsection{The surface patches}
The calculation in Lemma~\ref{lem:unique} also gives a description of the boundary near each modified edge. In the coordinates of Section~\ref{sec:arcs}, write $p=(x,y,z)$ and set
\[
F_p(u)=x^2+y^2+z^2+d^2-c^2d^2u^2+2zd^2u+2d|y|\sqrt{1-u^2}.
\]
Then $p\in\B(\gamma_B)$ if and only if
\begin{equation*}
\max_{h\le u\le 1}F_p(u)\le 1.
\end{equation*}
Equality holds precisely when $p$ has a supporting center on $\gamma_B$.

If the maximizing parameter lies in $(h,1)$, differentiating $F_p$ gives
\begin{equation*}
d^2(z-c^2u)=\frac{d|y|u}{\sqrt{1-u^2}}.
\end{equation*}
Together with $F_p(u)=1$, this equation describes the corresponding part of $\partial\B(\gamma_B)$ implicitly. The points that also satisfy the remaining ball inequalities belong to $\partial\E(R)$. Boundary portions supported by a fixed endpoint of $\gamma_B$ lie on the unit sphere centered at that endpoint.

There is also a description using circles. If $\gamma_B(s)$, with $|s|<t_0$, is a supporting center of $p$, then the distance from $p$ to $\gamma_B(s)$ is maximal at $s$. Consequently,
\[
\norm{p-\gamma_B(s)}=1, \quad (p-\gamma_B(s))\cdot\gamma_B'(s)=0.
\]
For fixed $s$, these equations describe a unit circle in the plane through $\gamma_B(s)$ perpendicular to $\gamma_B'(s)$. Its intersection with $\E(R)$ is exactly the set of boundary points having $\gamma_B(s)$ as a supporting center. So the boundary portions supported by interior points of $\gamma_B$ can be described by the corresponding portions of these circles.

\section{Final remarks}

\subsection{The tetrahedral case}
If $R$ is the Reuleaux tetrahedron, then $a=b=1$ for every dual pair. Thus
\[
c=d=\frac1{\sqrt2},\quad t_0=\frac\pi4,
\]
and each elliptic arc lies on an ellipse with semiaxes $1/\sqrt2$ and $1/2$, centered at the center of the tetrahedron. These are the six arcs used in \cite{MRR25}, so the resulting bodies coincide. As observed there, this constructed in \cite{ABPR25} in a different way. In particular, Theorem~\ref{thm:main} gives a three-dimensional proof of constant width for that body, while in \cite{MRR25} and \cite{ABPR25} it follows from the fact that a $4$-dimensional body has constant width.

\subsection{Combining elliptic and Meissner modifications}
The elliptic modification can also be combined with Meissner surgery. For each dual pair, we have three choices: round the first circular edge, round the second, or apply the elliptic modification to both.

To describe the resulting body as an intersection of balls form a set $S$ as follows. For each pair treated by Meissner surgery, include the circular edge that remains unmodified. For each pair treated by the elliptic modification, include both elliptic arcs. The mixed body is then $\B(S)$.

These choices can be made independently for different dual pairs. The comparison with Meissner bodies used in Lemma~\ref{lem:skeleton} gives $\diam S=1$. Moreover, the modifications are confined to the pairwise disjoint regions $D_e$. The local properties of Meissner surgery \cite{MR17}, together with the arguments of Section~\ref{sec:smoothness}, give $S\subset\partial\B(S)$ and smoothness at every boundary point outside $S$. Lemma~\ref{lem:criterion} therefore shows that the mixed body has constant width one.

To preserve all symmetries of $R$, we make these choices consistently on each orbit of dual pairs under its symmetry group. Choose a representative pair. If no symmetry of $R$ that preserves this pair exchanges its two edges, either choice extends consistently to the whole orbit. If a symmetry exchanges the two edges, use the elliptic modification throughout that orbit.

In particular, two edges of different lengths cannot be exchanged by a symmetry. For example, one way to preserve all symmetries is to use the elliptic modification on every equal-length pair and round the longer (or shorter) edge of every unequal-length pair.

We expect that combining Meissner and elliptic modifications will yield bodies of minimum volume among those of constant width $1$ invariant under a prescribed symmetry subgroup.

\subsection{Other symmetric constructions}
There is another way to obtain a body of constant width that preserves the symmetries of a given Reuleaux polyhedron. Simply construct two Meissner polyhedra from it and choose opposite edges to modify in each. Then take the Minkowski average of both, the resulting body preserves the symmetries and the vertices of the original Reuleaux polyhedron. From here a whole family of constant-width bodies can be constructed, since any weighted Minkwosky averages of bodies of constant width that preserve some symmetry group also has these properties.

\section*{Acknowledgments}
This work was supported by UNAM-PAPIIT grant IN114726.


\begin{thebibliography}{MPRR20}
\bibitem[AMO23]{AMO23}
I. Arelio, L. Montejano, and D. Oliveros,
\emph{Peabodies of constant width}, Beiträge zur Algebra und Geometrie
\textbf{64} (2023), 367--385.

\bibitem[AMO24]{AMO24}
I. Arelio, L. Montejano, and D. Oliveros,
\emph{Peabodies of constant width}, preprint (2024), arXiv:2408.13241.

\bibitem[ABPR25]{ABPR25}
A. Arman, A. Bondarenko, A. Prymak, and D. Radchenko,
\emph{Small volume bodies of constant width with tetrahedral symmetries},
SIGMA \textbf{21} (2025), Paper No. 109, 8 pp.

\bibitem[Hyn23]{Hyn23}
R. Hynd, \emph{The density of Meissner polyhedra}, preprint (2023), arXiv:2304.04035v2.

\bibitem[Hyn26]{Hyn26}
R. Hynd, \emph{A volume formula for Reuleaux polyhedra}, preprint (2026),
arXiv:2601.11756v1.

\bibitem[KMP10]{KMP10}
Y. S. Kupitz, H. Martini, and M. A. Perles,
\emph{Ball polytopes and the Vázsonyi problem}, Acta Mathematica Hungarica
\textbf{126} (2010), no. 1--2, 99--163.

\bibitem[MMO19]{MMO19}
H. Martini, L. Montejano, and D. Oliveros. \emph{Bodies of Constant Width: An Introduction to Convex Geometry with Applications}. Springer, Cham, 2019.

\bibitem[Mei18]{Mei18}
E. Meissner,
\emph{Über die durch reguläre Polyeder nicht stützbaren Körper}, Vierteljahresschr. Naturfor. Ges. Zürich
\textbf{63} (1918), 544--551.

\bibitem[MRR26]{MRR25}
M. G. Mercado-Flores, M. Raggi, and E. Roldán-Pensado,
\emph{A four-dimensional body of constant width}, Mathematika
\textbf{72} (2026), no. 3, e70113.

\bibitem[MPRR20]{MPRR20}
L. Montejano, E. Pauli, M. Raggi, and E. Roldán-Pensado,
\emph{The graphs behind Reuleaux polyhedra}, Discrete \& Computational
Geometry \textbf{64} (2020), no. 3, 1013--1022.

\bibitem[MR17]{MR17}
L. Montejano and E. Roldán-Pensado,
\emph{Meissner polyhedra}, Acta Mathematica Hungarica \textbf{151} (2017),
no. 2, 482--494.

\end{thebibliography}
\end{document}